\documentclass[a4paper,12pt,reqno]{amsart}

\usepackage{t1enc}
\usepackage[english]{babel}
\usepackage{amsmath,amssymb,eucal,amsthm}
\usepackage{bm}
\usepackage{graphicx}
\usepackage{mathrsfs}
\usepackage{mathptmx}
\usepackage{latexsym}
\usepackage{ulem}

\newtheorem{theorem}{Theorem}
\newtheorem{lemma}[theorem]{Lemma}

\theoremstyle{definition}
\newtheorem*{remark}{Remark}

\numberwithin{equation}{section}

\def\pp{\mathbb{P}}
\def\nn{\mathbb{N}}

\def\rr{\mathbb{R}}
\def\cc{\mathbb{C}}
\def\qq{\mathbb{Q}}

\def\gb{\mathfrak{B}}

\def\vf{\varphi}

\def\d{{\rm d}}
\def\meas{{\rm meas}}

\def\uZ{{\underline Z}}

\def\uO{{\underline \Omega}}
\def\uA{{\underline \alpha}}

\def\uhh{{\underline{h}}}

\def\uom{{\underline \omega}}
\def\ugb{{\underline \gb}}

\def\uH{{\underline H}}

\begin{document}
\hfill\texttt{\jobname.tex}\qquad\today

\bigskip
\title[Mixed joint discrete universality III]
{On mixed joint discrete universality for a class of zeta-functions III}

\author{Roma Ka{\v c}inskait{\.e} and Kohji Matsumoto}

\address{R. Ka{\v c}inskait{\.e} \\
Department of Mathematics and Statistics, Faculty of Informatics, Vytautas Magnus University, Kaunas, Vileikos 8, LT-44404, Lithuania}
\email{roma.kacinskaite@vdu.lt}

\address {K. Matsumoto \\
Graduate School of Mathematics, Nagoya University, Furocho, Chi\-ku\-sa-ku, Nagoya
464-8602, Japan}
\email{kohjimat@math.nagoya-u.ac.jp}

\date{}

\begin{abstract}
We present the most general at this moment results on the discrete mixed joint value-distribution (Theorems~\ref{RK-KM-gen-lim} and \ref{RK-KM-gen-lim-h})  and the universality property (Theo\-rems~\ref{RK-KM-general-h} and \ref{RK-KM-general}) for the class of Matsumoto zeta-functions and periodic Hurwitz zeta-functions under certain linear independence condition on the relevant parameters, such as common differences of arithmetic progressions, prime numbers etc.
\end{abstract}

\maketitle



\section{Introduction}\label{intro}

In analytic number theory, one of the most interesting and popular subjects is the so called universality property of various zeta- and $L$-functions in Voronin's sense (see \cite{SMV-1975}).    In particular the mixed joint universality is a rather hot topic today. Recall that the first results in this direction were obtained by J.~Sander and J.~Steuding   (see \cite{JST-JS-2006}) and independently by H.~Mishou (see \cite{HM-2007}). They proved that the pair of analytic functions is simultaneously approximable by the shifts of the Riemann zeta-function $\zeta(s)$ and the Hurwitz zeta-function $\zeta(s,\alpha)$ with certain real $\alpha$ (as it is well-known the function $\zeta(s)$ has the Euler type product expansion over primes, while the function $\zeta(s,\alpha)$ in general does not; from this facts the term ``mixed joint'' arises).

In the series of previous works by the authors (see \cite{RK-KM-2015}, \cite{RK-KM-2017-BAMS}, \cite{RK-KM-2017-Pal}, \cite{RK-KM-Pal-2018}
), the tuple consisting from the wide class of the Matsumoto zeta-functions $\varphi(s)$ and the periodic Hurwitz zeta-functions $\zeta(s,\alpha;\gb)$ is considered in the frame of the studies on the value distribution of such pair, as well as the universality property and the functional independence.

Let $s=\sigma+it$ be a complex variable and by $\pp, \nn, \qq, \rr$ and $\cc$ denote the set of all prime numbers, positive integers, rational numbers, real numbers and complex numbers, respectively. Suppose that $\alpha$, $0<\alpha\leq 1$, is a fixed real number. Recall the definitions of both the functions under our interest -- the Matsumoto zeta-functions ${\varphi}(s)$ and the periodic Hurwitz zeta-functions $\zeta(s,\alpha;\gb)$.

For $m\in\nn$, let $g(m)\in\nn$, $f(j,m)\in\nn$, $1\leq j\leq g(m)$.    Denote by $p_m$ the $m$th prime number, and let $a_m^{(j)}\in\cc$.
Assume that $g(m)\leq C_1 p_m^{\alpha}$ and $|a_m^{(j)}|\leq p_m^{\beta}$ with positive constant $C_1$ and non-negative  constants $\alpha$ and $\beta$.
Define the zeta-function ${\widetilde\varphi}(s)$ by the polynomial Euler product
$$
\widetilde{\varphi}(s)=\prod_{m=1}^{\infty}\prod_{j=1}^{g(m)}\left(1-a_m^{(j)} p_m^{-sf(j,m)}\right)^{-1}.
$$
The function $\widetilde{\varphi}(s)$ converges absolutely in the region $\sigma>\alpha+\beta+1$, and, in this region, it can be written as the Dirichlet series
$$
\widetilde{\varphi}(s)=\sum_{k=1}^{\infty}\frac{{\widetilde{c}}_k}{ k^{s+\alpha+\beta}}
$$
with coefficients ${\widetilde{c}}_k$ such that ${\widetilde{c}}_k=O\big(k^{\alpha+\beta+\varepsilon}\big)$  if all prime factors of $k$ are large for every positive $\varepsilon$ (for the comment, see \cite{RK-KM-2017-BAMS}).

For brevity, denote the shifted version of the function $\widetilde{\varphi}(s)$ as 
$$
\varphi(s):=\widetilde{\varphi}(s+\alpha+\beta)=\sum_{k=1}^{\infty}\frac{{\widetilde{c}}_k}{ k^{s+\alpha+\beta}}=\sum_{k=1}^{\infty}\frac{{c}_k}{ k^{s}}
$$
with $c_k:={\widetilde{c}}_k k^{-\alpha-\beta}$.   This series is convergent in the half-plane  $\sigma>1$.

Suppose that the function $\varphi(s)$ satisfies the following conditions:
\begin{enumerate}
	\item[(i)] $\varphi(s)$ can be continued meromorphically to $\sigma\geq\sigma_0$, where $\frac{1}{2}\leq\sigma_0<1$, and all poles of $\varphi(s)$ in this region
	are included in a
	compact set which has no intersection with the line $\sigma=\sigma_0$;
	\item[(ii)] for $\sigma\geq\sigma_0$, $\varphi(\sigma+it)=O(|t|^{C_2})$
	with a certain $C_2>0$;
	\item[(iii)] it holds the mean-value estimate
	\begin{align}\label{meanvalue-for-varphi}
	\int_0^T |\varphi(\sigma_0+it)|^2 dt =O(T).
	\end{align}
	.
\end{enumerate}

\noindent We call the set of all such functions $\varphi(s)$ as the class of Matsumoto zeta-functions and denote by $\mathcal{M}$.

Now, for $\mathbb{N}_0:=\mathbb{N}\cup \{0\}$, let  $\mathfrak{B}=\{ b_m: m \in \mathbb{N}_0\}$  be a periodic sequence of complex numbers $b_{m}$ with minimal positive period $l\in \mathbb{N}$. The periodic Hurwitz zeta-function $\zeta(s,\alpha; \mathfrak{B})$ with parameter $\alpha \in \rr$, $0<\alpha\leq 1$, is given  by the Dirichlet series
$$
\zeta(s,\alpha; \mathfrak{B})= \sum_{m=0}^{\infty} \frac{{b_m}}{{(m+\alpha)^{s}}} \quad \text{for} \quad \sigma>1.
$$
It is known that the function $\zeta(s,\alpha; \mathfrak{B})$ is analytically continued to the whole complex plane, except for a possible simple pole at the point $s=1$ with residue
$
b:= \frac{1}{l} \sum_{m=0}^{l-1} b_m.
$
If $b=0$, then $\zeta(s,\alpha; \mathfrak{B})$ is an entire function.

It is possible to prove functional limit theorems for the whole class $\mathcal M$, but it is  difficult to prove the denseness lemma which is necessary to prove the universality theorem. Therefore in the proof of the universality we use an assumption that the function $\varphi(s)$ belongs to the Steuding class ${\widetilde{S}}$ defined below.

We say that the function $\varphi(s)$ belongs to the class ${\widetilde{S}}$ if the following conditions are satisfied:
\begin{itemize}
	\item[(a)] there exists a Dirichlet series expansion
	$
	\varphi(s)=\sum_{m=1}^{\infty}a(m)m^{-s}
	$
	with $a(m)=O(m^\varepsilon)$ for every $\varepsilon>0$;
	\item[(b)] there exists $\sigma_\varphi<1$ such that $\varphi(s)$ can be meromorphically continued to the half-plane $\sigma>\sigma_\varphi$;
	\item[(c)] for every fixed $\sigma>\sigma_\varphi$ and $\varepsilon>0$, the order estimate
	$
	\varphi(\sigma+it)=O(|t|^{C_3+\varepsilon})
	$
	with $C_3\geq 0$ holds;
	\item[(d)] there exists the Euler product expansion over primes, i.e.,
	$$
	\varphi(s)=\prod_{p \in \mathbb{P}}\prod_{j=1}^{J}\left(1-\frac{{a_j(p)}}{{p^{s}}}\right)^{-1};
	$$
	\item[(e)] there exists a constant $\kappa>0$ such that
	$$
	\lim_{x \to \infty}\frac{1}{\pi(x)}\sum_{p \leq x}|a(p)|^2=\kappa,
	$$
	where $\pi(x)$ denotes the number of primes $p$ up to $x$.
\end{itemize}

For $\varphi(s)\in {\widetilde{S}}$, let $\sigma^*$ be the infimum of all $\sigma_1$ such that
$$
\frac{1}{2T}\int_{-T}^{T}|\vf(\sigma+it)|^2 \d t \sim \sum_{m=1}^{\infty}\frac{|a(m)|^2}{m^{2\sigma}}
$$
holds for any $\sigma \geq \sigma_1$. Then $\frac{1}{2}\leq \sigma^*<1$, and we see that ${\widetilde{S}} \subset\mathcal{M}$ if  $\sigma_0=\sigma^*+\varepsilon$ is chosen. Note that the class $\widetilde S$ is not a subclass of the Selberg class!

To formulate mixed joint discere limit theorems and universality property  for the tuple $\big(\varphi(s),\zeta(s,\alpha;\gb)\big)$ and further results, we need some notation. For any compact set $K \subset \cc$, denote by $H^c(K)$ the set of all $\cc$-valued functions defined on $K$, continuous on $K$ and holomorphic in the interior of $K$. By $H_0^c(K)$ denote the subset of $H^c(K)$ such that, on $K$, all elements of $H^c(K)$  are non-vanishing.
Let $D(a,b)=\{s \in \cc: \; a<\sigma<b\}$ for $a,b \in \rr$, $a<b$, and we denote by $\meas \{A\}$ the Lebesgue measure of the measurable set $A\subset \rr$.
For any set $S$,  ${\mathcal{B}}(S)$ denotes the set of all Borel subsets of $S$, and, for any region $D$, $H(D)$ denotes the set of all holomorphic functions on $D$.

The first result on the mixed joint universality of the pair $\big(\varphi(s),\zeta(s,\alpha;\gb)\big)$ is the following theorem, which considers the situation when the shifting parameter is moving continuously.

\begin{theorem}[Theorem 2.2, \cite{RK-KM-2015}]\label{RK-KM-univ-2015}
Suppose $\varphi\in\widetilde{S}$, and $\alpha$ is a transcendental number, $0<\alpha <1$. Let $K_1$ be a compact subset of the strip $D\big(\sigma^*,1\big)$, $K_2$ be a
compact subset of the strip $D\big(\frac{1}{2},1\big)$, both with connected complements. Suppose that $f_1(s) \in H_0^c(K_1)$, $f_2(s) \in H^c(K_2)$. Then, for every $\varepsilon>0$, it holds that
	\begin{align*}
	\liminf_{T \to \infty}\frac{1}{T}\meas \bigg\{\tau \in [0,T]: & \sup_{s \in K_1}|\varphi(s+i\tau)-f_1(s)|<\varepsilon, \cr
	&\sup_{s \in K_2}|\zeta(s+i\tau,\alpha;\gb)-f_2(s)|<\varepsilon\bigg\}>0.
	\end{align*}
\end{theorem}

More interesting and complicated questions on mixed joint universality are concerning the discrete case, when shifting parameters take discrete values.   The first result of such kind is the discrete analogue of Theorem~\ref{RK-KM-univ-2015}, which was proved by the authors in \cite{RK-KM-2017-Pal}.

Let $h>0$ be the shifting parameter, that is the common difference of the relevant arithmetic progression, and put
\begin{align*}
L(\alpha,h):=\left\{\log p: p \in \pp\right\}\cup\left\{\log (m+\alpha): m \in \nn_0\right\}\cup\left\{\frac{2\pi}{h}\right\}.
\end{align*}

\begin{theorem}[Theorem~3, \cite{RK-KM-2017-Pal}]\label{RK-KM-disc-univ}
	Suppose $\varphi\in\widetilde{S}$, and that the elements of the set $L(\alpha,h)$ are linearly independent over $\qq$.
	Let $K_1$, $K_2$, $f_1(s)$, $f_2(s)$ be the same as in Theorem~\ref{RK-KM-univ-2015}.
	Then, for every $\varepsilon>0$, it holds that
	\begin{align*}
	\liminf_{N \to \infty}\frac{1}{N+1}\# \bigg\{0\leq k\leq N: & \sup_{s \in  K_1}|\varphi(s+ikh)-f_1(s)|<\varepsilon, \cr
	&\sup_{s \in K_2}|\zeta(s+ikh,\alpha;\gb)-f_2(s)|<\varepsilon\bigg\}>0.
	\end{align*}
\end{theorem}
\noindent Here and in what follows, $\# \{A\}$ denotes the cardinality of the set $A$.

In \cite{RK-KM-Pal-2018}, we have extended our investigations and studied the case when common differences of arithmetic progressions for both of the zeta-functions in the tuple are dif\-fe\-rent.

The aim of this paper is to show more general results than those mentioned above. Here we prove two mixed joint universality theorems -- in the cases of the same and different common differences of arithmetic progressions.
The novelty  is that a wide collection of the periodic Hurwitz zeta-functions 
$\zeta(s,\alpha_j;\gb_{jl})$ will be constructed; here for each $\alpha_j$ a collection 
of sequences $\gb_{jl}$ is attached.   This type of general collection of periodic
Hurwitz zeta-functions has been studied in several previous articles (such as
\cite{AL-2010}, \cite{AL-SS-2009}, \cite{JG-RM-SR-DS-2010} and \cite{RK-KM-2017-BAMS}) 
in the continuous case, 
but it seems that there is no previous work in the discrete case.
In fact, the results in the present article are discrete analogous of Theorem~4.2 from \cite{RK-KM-2017-BAMS}, but in \cite{RK-KM-2017-BAMS} we assume the stronger hypothesis,
that is the  algebraic independence of the parameters $\alpha_1,...,\alpha_{r}$.

\medskip

Suppose that $l(j)$ is a positive integer, $j=1,\ldots,r$, and $\lambda=l(1)+\ldots+l(r)$. For each $j$ and $l$, $j=1,...,r$, $l=1,...,l(j)$, let
${\mathfrak{B}}_{jl}=\{b_{mjl}\in\mathbb{C}: m \in \mathbb{N}_0\}$ be a periodic sequence of complex numbers  $b_{mjl}$ (not all zero) with the minimal period $k_{jl}$. Suppose that $\alpha_j$ be a real number such that $\alpha_j\in (0,1)$, $j=1,\ldots,r$.
Let $\zeta(s,\alpha_j;{\mathfrak{B}}_{jl})$ be the corresponding periodic Hurwitz zeta-function for $j=1,...,r$, $l=1,...,l(j)$.
Moreover, let $k_j$ the least common multiple of periods $k_{j1},\ldots,k_{jl(j)}$, and
$$
B_j:=\begin{pmatrix}
b_{1j1}  & b_{1j2} & \ldots & b_{1jl(j)}\cr
b_{2j1}  & b_{2j2} & \ldots & b_{2jl(j)}\cr
\ldots & \ldots & \ldots  & \ldots \cr
b_{k_j j1}  & b_{k_j j2} & \ldots & b_{k_j jl(j)}\cr
\end{pmatrix}, \quad j=1,...,r.
$$

Now let $\uA:=(\alpha_1,...,\alpha_r)$, $h>0$, and
let $\uhh:=\big(h_1,h_{21},...,h_{2r}\big)$ for $h_1>0$ and $h_{2j}>0$ 
with $j=1,..,r$.   Define two sets
\begin{align*}
L(\uA,h) :=&L\big(p,\alpha_1,...,\alpha_r,h\big)\cr
=&\{\log p: p \in \pp\}\cup\{\log(m+\alpha_j): m\in \nn_0, j=1,...,r\}\cup\left\{\frac{2\pi}{h}\right\},
\end{align*}
and
\begin{align*}
L(\uA,\uhh) :=&L\big(p,\alpha_1,...,\alpha_r,h_1,h_{21},...,h_{2r}\big)\cr
=&\{h_1 \log p: p \in \pp\}\bigcup_{j=1}^{r}\{h_{2j}\log(m+\alpha_j): m\in \nn_0\}\cup\{\pi\}.
\end{align*}

Our first new discrete mixed joint universality theorem is as follows.
\begin{theorem}\label{RK-KM-general-h}
	Suppose that the elements of the set $L(\uA,h)$ are linearly independent over ${\mathbb{Q}}$, ${\rm rank}B_j=l(j)$, $j=1,...,r$, and $\varphi(s)$ belongs to the class ${\widetilde S}$.
	Let $K_1$ be a compact subset of $D(\sigma^*,1)$, $K_{2jl}$ be compact subsets of 	$D(\frac{1}{2},1)$, $l=1,...,l(j)$,	all of them with connected complements.
	Suppose that $f_1\in H_0^c(K_1)$ and $f_{2jl}\in H^c(K_{2jl})$.
	Then, for every $\varepsilon>0$, it holds that
	\begin{align*}
	\liminf\limits_{N \to \infty}\frac{1}{N+1}\bigg\{0\leq k\leq N: &
	\sup\limits_{s \in K_1}|\varphi(s+ikh)-f_1(s)|<\varepsilon, \\
	& \sup\limits_{1\leq j\leq r}\sup\limits_{1\leq l\leq l(j)}
	\sup\limits_{s\in K_{2jl}}|\zeta(s+ikh,\alpha_j;{\mathfrak{B}}_{jl})-f_{2jl}(s)|<
	\varepsilon\bigg\}>0.
	\end{align*}
\end{theorem}

In Theorem~\ref{RK-KM-general-h} there appears only one common difference $h$.
Our second new theorem on universality describes the situation when the common differences
associated with relevant zeta-functions can be different from each other.

\begin{theorem}\label{RK-KM-general}
	Suppose that the elements of the set $L(\uA,\uhh)$ are linearly independent over ${\mathbb{Q}}$, and $B_j$, $f_1(s), f_{2jl}(s)$, $K_1$, $K_{2jl}$ and $\varphi(s)$ satisfy hypotheses of Theorem~\ref{RK-KM-general-h}.
	Then, for every $\varepsilon>0$, it holds that
	\begin{align*}
	\liminf\limits_{N \to \infty}\frac{1}{N+1}\bigg\{0\leq k\leq N: &
	\sup\limits_{s \in K_1}|\varphi(s+ikh_1)-f_1(s)|<\varepsilon, \\
	& \sup\limits_{1\leq j\leq r}\sup\limits_{1\leq l\leq l(j)}
	\sup\limits_{s\in K_{2jl}}|\zeta(s+ikh_{2j},\alpha_j;{\mathfrak{B}}_{jl})-f_{2jl}(s)|<
	\varepsilon\bigg\}>0.
	\end{align*}
\end{theorem}

\begin{remark}\label{remark}
	As we already have noted, Theorems~\ref{RK-KM-general-h} and \ref{RK-KM-general} are  discrete analogues of Theo\-rem~4.2 from \cite{RK-KM-2017-BAMS}, but here, instead of algebraic independence of the parameters $\alpha_1,...,\alpha_{r}$, the linear independence of the elements of the sets $L(\uA,h)$ and $L(\uA,\uhh)$ over the set of rational numbers $\qq$ are used, respectively.
\end{remark}
	
\section{A discrete limit theorem}\label{aux-lemmas}

Theorem~\ref{RK-KM-general-h} is obviously a special case of
Theorem~\ref{RK-KM-general}, so it is enough to prove Theorem~\ref{RK-KM-general}.
For this aim, 
we use a discrete mixed joint limit theorem (Theorem \ref{RK-KM-gen-lim} below)
in the sense of weakly convergent probability measures in the space of analytic functions. 
Since Theorem \ref{RK-KM-gen-lim} is valid for more general $\varphi\in{\mathcal M}$, 
we will formulate the theorem in such a general setting.
In this section we first prove two auxiliary lemmas which play essential roles in 
the proof of Theorem \ref{RK-KM-gen-lim}.

We introduce certain topological structure.
Let $\gamma$ be the unit circle on the complex plane, and $\Omega_1$ and $\Omega_2$ be two tori defined as 
\begin{align*}
\Omega_1=\prod\limits_{p\in\mathbb{P}}\gamma_p \quad \text{and}\quad \Omega_2=\prod\limits_{m=0}^{\infty}\gamma_m,
\end{align*}
where $\gamma_p=\gamma$ for all $p\in \mathbb{P}$, and $\gamma_m=\gamma$ for all
$m \in \mathbb{N}_0$, respectively. It is well known that both tori $\Omega_1$ and $\Omega_2$ are
compact topological Abelian groups with respect to the product topology and pointwise multiplication. Therefore,
\begin{align*}
{\underline \Omega}:=\Omega_1\times \Omega_{21}\times ... \times \Omega_{2r}
\end{align*}
with $\Omega_{2j}=\Omega_2$ for all $j=1,...,r$ also is a compact topological group. This gives that we can define the probability Haar measure $m_H$ on $({\underline \Omega}, {\mathcal{B}}({\underline \Omega}))$. This fact allows us to have the probability space $({\underline \Omega},{\mathcal B}({\underline \Omega}), {m}_H)$ (see \cite{JST-2007}).

Denote by $\omega_1(p)$ be the projection of $\omega_1 \in \Omega_1$ to $\gamma_p$, $p \in {\mathbb{P}}$, and by $\omega_{2j}(m)$ the projection of $\omega_{2j} \in \Omega_{2j}$ to $\gamma_m$, $m \in {\mathbb{N}}_0$, $j=1,...,r$. 
The definition of $\omega_1(m)$ for general $m\in\nn_0$ is given by
$\omega_1(m)=\omega_1(p_1)^{a_1}\cdots\omega_1(p_r)^{a_r}$ according to the 
decomposition $m=p_1^{a_1}\cdots p_r^{a_r}$ into prime divisors.
Let ${{\underline \omega}}=(\omega_1,\omega_{21},..., \omega_{2r})$ be an element of $\underline \Omega$.

For $A \in {\mathcal{B}}(\underline \Omega)$, on $({\underline \Omega}, {\mathcal{B}}({\underline \Omega}))$, define
\begin{align*}
Q_N(A):=\frac{1}{N+1}\#\bigg\{& 0 \leq k \leq N:  \bigg(\big(p^{-ikh_1}: p \in \pp\big),
\cr &\big((m+\alpha_1)^{-ikh_{21}}: m\in \nn_0\big),...,\big((m+\alpha_r)^{-ikh_{2r}}: m\in \nn_0\big)
\bigg)
\in A\bigg\}.
\end{align*}

\begin{lemma}\label{on-torus}
Suppose that the elements of the set $L({\underline\alpha},\uhh)$ are linearly independent over $\qq$. Then $Q_N$ converges weakly to the Haar measure $m_H$ as $N \to \infty$.
\end{lemma}

\begin{proof}
This is Lemma~6 from \cite{RK-KM-Pal-2018}.
\end{proof}

Another important auxiliary result is related with the ergodic theory. It will be used for the identification of the explicit form of the limit measure.

Let $a_{\uA,\uhh}:=\{(p^{-ih_1}:p \in \pp),((m+\alpha_1)^{-ih_{21}}: m \in \nn_0), ..., ((m+\alpha_r)^{-ih_{2r}}: m \in \nn_0)\}$ be an element of $\uO$. Since the Haar measure $m_H$ is an invariant with respect to translations on $\uO$, 
we define the measurable measure-preserving transformation of the torus $\uO$ by the formula
\begin{align*}
\Phi_{\uA,\uhh}(\uom):=a_{\uA,\uhh}\uom \qquad (\uom\in \uO).
\end{align*}
Recall that the set $A \in {\mathcal B}(\uO)$ is invariant with respect to the transformation $\Phi_{\uA,\uhh}$  if the sets $A$ and $\Phi_{\uA,\uhh}(A)$ differ from each other at most by the set of zero $m_H$-measure, and the transformation $\Phi_{\uA,\uhh}$ is ergodic if its Borel $\sigma$-field of invariant sets consists of sets having $m_H$-measure 0 or 1.

\begin{lemma}\label{ergodic}
	Suppose that the elements of the set $L(\uA,\uhh)$ are linearly independent over 
$\qq$. Then the transformation $\Phi_{\uA,\uhh}$ is ergodic.
\end{lemma}

\begin{proof}
This lemma is proved by the Fourier transform method.
Any character $\chi$ of the torus $\underline \Omega$ can be written in the form
$$
\chi(\uom)=\prod_{p\in\pp}\omega_1^{k_p}(p)\prod_{m\in \nn_0}\prod_{j=1}^r 
\omega_{2j}^{l_{mj}}(m),
$$
where $k_p,l_{m1},\ldots,l_{mr}$ are integers, only a finite number of which are distinct
from 0.    Then
$$
\chi(a_{\uA,\uhh})=\exp\left(-ih_1\sum_{p\in\pp}k_p \log p-i\sum_{m\in\nn_0}\sum_{j=1}^r
h_{2j}l_{mj}\log(m+\alpha_j)\right)\neq 1
$$
for any non-trivial $\chi$, because of the linear independence of the elements of
$L(\uA,\uhh)$.    Using this fact, we proceed along the standard way
to prove the lemma; see, for example,
Section 4 of \cite{EB-AL-2015-LMJ}.
\end{proof}

\medskip

Now we are ready for the discrete mixed joint limit theorem for the tuple of the class of zeta-functions under our investigation.

Let $D_1$ be an open subset of $D(\sigma^*,1)$, and $D_2$ be an open subset of
$D\big(\frac{1}{2},1\big)$.
Denote the Cartesian product of $\lambda+1$ such spaces as
$$
{\underline H}=H(D_1)\times \underbrace{H(D_2)\times ...\times H(D_2)}\limits_{\lambda}.
$$

Let $\varphi\in{\mathcal M}$, and define
\begin{align*}
{{\underline Z}}(\underline{s},{{\underline \alpha}};{{\underline{\mathfrak{B}}}})=\bigg(\varphi(s_1),&\zeta(s_{211},\alpha_1;{\mathfrak{B}}_{11}),...,\zeta(s_{21l(1)},\alpha_1;{\mathfrak{B}}_{1l(1)}),...,\cr
& \zeta(s_{2r1},\alpha_r;{\mathfrak{B}}_{r1}),...,\zeta(s_{2rl(r)},\alpha_r;{\mathfrak{B}}_{r l(r)})
\bigg)
\end{align*}
with $\underline{\mathfrak{B}}=\big(\mathfrak{B}_{11},\ldots, \mathfrak{B}_{1l(1)},
\cdots,$ $\mathfrak{B}_{r1},\cdots,\mathfrak{B}_{rl(r)}\big)$ and
$\underline{s}=(s_1,s_{211},\ldots,s_{21l(1)},\ldots,s_{2r1},\ldots,$ $s_{2rl(r)})
\in\mathbb{C}^{\lambda+1}$.   Further we write
$$
\underline{s}+ik\underline{h}=(s_1+ikh_1,s_{211}+ikh_{21},\ldots,s_{21l(1)}+ikh_{21},\ldots,s_{2r1}+ikh_{2r},\ldots, s_{2rl(r)}+ikh_{2r})
$$
for $k\in\nn_0$.
The main tool for the proof of our discrete mixed universality theorems is the limit theorem on $(\underline{H},{\mathcal B}(\underline{H}))$ for
$$
P_N(A):=\frac{1}{N+1}\#\big\{0 \leq k \leq N: \underline{Z}(\underline{s}
+ik\underline{h},\underline{\alpha};\underline{\mathfrak{B}})\in A\big\}, \quad A \in {\mathcal B}(\underline{H}).
$$

On  $(\uO,{\mathcal B}(\uO),m_H)$, define  the ${{\underline H}}(D)$-valued random element ${{\underline Z}}(\underline{s},{{\underline \omega}}, {{\underline \alpha}};{{\underline{\mathfrak{B}}}})$ by the formula
\begin{align*}
{{\underline Z}}(\underline{s},{{\underline \omega}}, {{\underline \alpha}};{{\underline{\mathfrak{B}}}})=\bigg(\varphi(s_1,\omega_1), & \zeta(s_{211},\alpha_1,\omega_{21};{\mathfrak{B}}_{11}),...,\zeta(s_{21l(1)},\alpha_1,\omega_{21};{\mathfrak{B}}_{1l(1)}),...,\cr
& \zeta(s_{2r1},\alpha_r,\omega_{2r};{\mathfrak{B}}_{r1}),...,\zeta(s_{2rl(r)},\alpha_r,\omega_{2r};{\mathfrak{B}}_{r l(r)})
\bigg)
\end{align*}
with
\begin{align*}
\varphi(s_1,\omega_1)=\sum_{m=1}^{\infty}\frac{c_m\omega_1(m)}{{m^{s_1}}}, \quad s_1 \in D(\sigma^*,1),
\end{align*}
and
\begin{align*}
\zeta(s_{2jl},\alpha_j,\omega_{2j};\mathfrak{B}_{jl})=\sum_{m=0}^{\infty}\frac{b_{mjl}\omega_{2j}(m)}{(m+\alpha_j)^{s_{2jl}}}, \quad  s_{2jl} \in D\left(\frac{1}{2},1\right), \quad j=1,...,r, \quad l=1,...,l(j).
\end{align*}
These series are convergent for almost all $\omega_1 \in \Omega_1$ and $\omega_{2j}\in \Omega_{2j}$, $j=1,...,r$. Let $P_{{{\underline Z}}}$ be the distribution of the random element ${{\underline Z}}({\underline s},{{\underline \omega}}, {{\underline \alpha}};{{\underline{\mathfrak{B}}}})$, i.e., the probability measure on $({{\underline H}},{\mathcal B}({{\underline H}}))$ defined as
\begin{align*}
P_{{{\underline Z}}}(A)=m_H \big(
{{\underline \omega}} \in \uO: {{\underline Z}}({\underline s},{{\underline \omega}},{{\underline \alpha}};{{\underline{\mathfrak{B}}}}) \in A
\big), \quad A \in {\mathcal{B}}({\underline H}).
\end{align*}

\begin{theorem}\label{RK-KM-gen-lim}
	Suppose that $\varphi\in\mathcal{M}$, and the elements of the set $L(\uA,\uhh)$ are linearly independent over $\qq$. 
Then $P_N$ converges weakly to
	$P_{\uZ}$
as $N \to \infty$.
\end{theorem}

We prove this theorem in a way whose principle has been already well-developed (see, for example, Theorem~4 of \cite{RK-KM-Pal-2018} or Lemma~5.1 of \cite{RK-KM-2017-BAMS} as the continuous analogue). 
However our present situation includes a lot of parameters which may cause a trouble
for the readers, so we will explain many details even if they are now standard.

\section{Proof of the limit theorem}

We begin with a discrete mixed joint limit theorem for absolutely convergent Dirichlet series.
For a fixed number $\sigma_0^*>\frac{1}{2}$, let
$$
v_1(m,n):=\exp\left(-\left(\frac{m}{n}\right)^{\sigma_0^*}\right),\quad m\in\nn, \quad n \in \nn,
$$
and
$$
v_2(m,n,\alpha_j)=\exp\left(-\left(\frac{m+\alpha_j}{n+\alpha_j}\right)^{\sigma_0^*}\right),
\quad m\in\nn_0, \quad n \in \nn, \quad j=1,...,r.
$$

In view of the Mellin transform formula and contour integration, we can show 
(cf. pp.153-154 of \cite{AL-1996-Bordo}, p.87 of \cite{AL-RG-2002}) that the series
$$
\varphi_n(s_1)=\sum_{m=1}^{\infty}\frac{c_mv_1(m,n)}{m^{s_1}},\quad s_1 \in D(\sigma^*,1),
$$
and
$$
\zeta_n(s_{2jl},\alpha_j;\gb_{jl})=\sum_{m=0}^{\infty}\frac{b_{mjl}v_2(m,n,\alpha_j)}{(m+\alpha_j)^{s_{2jl}}}, \quad s_{2jl}\in D\left(\frac{1}{2},1\right),\quad j=1,...,r, \quad l=1,...,l(r),
$$
converge absolutely when the real parts of all complex variables are 
greater than $\frac{1}{2}$.

Moreover, for $\uom\in \uO$, let
$$
\varphi_n(s_1,\omega_1)=\sum_{m=1}^{\infty}\frac{c_m\omega_1(m) v_1(m,n)}{m^{s_1}},\quad s_1 \in D(\sigma^*,1),
$$
and
$$
\zeta_n(s_{2jl},\alpha_j,\omega_{2j};\gb_{jl})=\sum_{m=0}^{\infty}\frac{b_{mjl} \omega_{2j}(m)v_2(m,n,\alpha_j)}{(m+\alpha_j)^{s_{2jl}}}, \quad s_{2jl}\in D\left(\frac{1}{2},1\right),
$$
$j=1,...,r$, $l=1,...,l(r)$.
These series are convergent for almost all $\omega_1 \in \Omega_1$ and $\omega_{2j} \in \Omega_{2j}$, $j=1,...,r$, when the real parts of all complex variables are greater 
than $\frac{1}{2}$.

Now let us fix ${\widehat\uom}=\big({\widehat\omega}_1,{\widehat\omega}_{21},..., {\widehat\omega}_{2r}\big)\in \uO$. For $A \in {\mathcal B}(\uH)$, define two measures
\begin{align*}
P_{N,n}(A):=\frac{1}{N+1}\#\bigg\{0 \leq k \leq N: {{\underline Z}}_n(\underline{s}
+ik\underline{h},{{\underline \alpha}};{{\underline{\mathfrak{B}}}}) \in A
\bigg\}
\end{align*}
and
\begin{align*}
P_{N,n,{\widehat \uom}}(A):=\frac{1}{N+1}\#\bigg\{0 \leq k \leq N: {{\underline Z}}_n(\underline{s}+ik\underline{h},{{\underline {\widehat\omega}}}, {{\underline \alpha}};{{\underline{\mathfrak{B}}}}) \in A\bigg\},
\end{align*}
where
\begin{align*}
{{\underline Z}}_n(\underline{s},{{\underline \alpha}};{{\underline{\mathfrak{B}}}})=\bigg(\varphi_n(s_1),&\zeta_n(s_{211},\alpha_1;{\mathfrak{B}}_{11}),...,\zeta_n(s_{21l(1)},\alpha_1;{\mathfrak{B}}_{1l(1)}),...,\cr
& \zeta_n(s_{2r1},\alpha_r;{\mathfrak{B}}_{r1}),...,\zeta_n(s_{2rl(r)},\alpha_r;{\mathfrak{B}}_{r l(r)})
\bigg)
\end{align*}
and
\begin{align*}
{{\underline Z}}_n(\underline{s},{{\underline {\widehat \omega}}}, {{\underline \alpha}};{{\underline{\mathfrak{B}}}})=\bigg(\varphi_n(s_1, {\widehat\omega}_1), & \zeta_n(s_{211},\alpha_1, {\widehat\omega}_{21};{\mathfrak{B}}_{11}),...,\zeta_n(s_{21l(1)},\alpha_1,{\widehat\omega}_{21};{\mathfrak{B}}_{1l(1)}),...,\cr
& \zeta_n(s_{2r1},\alpha_r, {\widehat\omega}_{2r};{\mathfrak{B}}_{r1}),...,\zeta_n(s_{2rl(r)},\alpha_r, {\widehat\omega}_{2r};{\mathfrak{B}}_{r l(r)})
\bigg).
\end{align*}


\begin{lemma}\label{RK-lem-1}
	Suppose that the elements of the set $L(\uA,\uhh)$ are linearly independent over $\qq$.  Then for all $n$, both the measures $P_{N,n}$ and $P_{N,n,{\widehat \omega}}$ converge weakly to the same probabi\-li\-ty measure (denote it by $P_n$) on $(\underline{H},$ ${\mathcal B}(\underline{H}))$ as $N \to \infty$.
\end{lemma}

\begin{proof}
This is a generalization of Lemma 1 of \cite{RK-KM-Pal-2018}.   
The following proof is similar to the argument included in several previous articles, such as
\cite{JG-RM-SR-DS-2010}, \cite{AL-2014}.

Define $h_n:\underline{\Omega}\to\underline{H}$ by 
$h_n(\uom)={{\underline Z}}_n(\underline{s},{{\underline {\omega}}}, {{\underline \alpha}};{{\underline{\mathfrak{B}}}})$.
This is continuous, and 
\begin{align*}
&h_n\left((p^{-ikh_1}:p \in \pp),((m+\alpha_1)^{-ikh_{21}}: m \in \nn_0), ..., ((m+\alpha_r)^{-ikh_{2r}}: m \in \nn_0)\right)\\
&\qquad\qquad ={{\underline Z}}_n(\underline{s}+ik\underline{h}, {{\underline \alpha}};{{\underline{\mathfrak{B}}}}).
\end{align*}
Therefore $P_{N,n}=Q_N\circ h_n^{-1}$.    This fact, Lemma \ref{on-torus} and Theorem 5.1 of \cite{PB-1968} gives that $P_{N,n}$ converges weakly to $m_H\circ h_n^{-1}$ as $N\to\infty$. 

Next define $g_n:\underline{\Omega}\to\underline{H}$ by 
$g_n(\uom)=h_n(\uom\cdot \widehat{\uom})$.    Then
\begin{align*}
&g_n\left((p^{-ikh_1}:p \in \pp),((m+\alpha_1)^{-ikh_{21}}: m \in \nn_0), ..., ((m+\alpha_r)^{-ikh_{2r}}: m \in \nn_0)\right)\\
&\qquad\qquad ={{\underline Z}}_n(\underline{s}+ik\underline{h}, \widehat{\uom},
{{\underline \alpha}};{{\underline{\mathfrak{B}}}}),
\end{align*}
so $P_{N,n,\widehat{\uom}}=Q_N\circ g_n^{-1}$ which converges weakly to $m_H\circ g_n^{-1}$
as above.   
Therefore the lemma follows, because the invariance property of the Haar measure implies that
$m_H\circ g_n^{-1}=m_H\circ h_n^{-1}$.    
\end{proof}

\medskip

In the next step of the proof, we pass from ${\underline Z}_n({\underline s}, \uA;\ugb)$ and ${\underline Z}_n({\underline s}, { \uom}, \uA;\ugb)$ to   $\underline{Z}({\underline s}, \uA;\ugb)$ and $\underline{Z}({\underline s}, \uom, \uA;\ugb)$ by the approximation in mean, respectively.
First define a metric ${\underline \varrho}$ on $\underline{H}$. 
For any open region $G\subset \cc$, let $\rho_G$ be the standard metric on $H(G)$ which
induces the topology of uniform convergence on compact subsets; see Section 1.7 of
\cite{AL-1996}.    Then, for two elements
\begin{align*}
&\underline{f}=(f_1,f_{211},\ldots,f_{21l(1)},\ldots,f_{2r1},\ldots,
f_{2rl(r)}),\\
&\underline{g}=(g_1,g_{211},\ldots,g_{21l(1)},\ldots,g_{2r1},\ldots,
g_{2rl(r)}), 
\end{align*}
of $\underline{H}$, define the metric ${\underline \varrho}$ by
$$
{\underline \varrho}(\underline{f},\underline{g})=\max\left\{\rho_{D_1}(f_1,g_1),
\max_{1\leq j\leq r,1\leq l\leq l(j)}\rho_{D_2}(f_{2jl},g_{2jl})\right\}.
$$

\begin{lemma}\label{RK-lem-2}
	Let  $L(\uA,\uhh)$ be as in Theorem~\ref{RK-KM-gen-lim}. Then we have
	\begin{align*}
	\lim_{n \to \infty}\limsup_{N \to \infty}\frac{1}{N+1}\sum_{k=0}^{N}
	{\underline \varrho}\big(\underline{Z}({\underline s}+ik\underline{h}, \uA;\ugb),\underline{Z}_n({\underline s}+ik\underline{h}, \uA;\ugb)\big)=0
	\end{align*}
	and, for almost all $\uom \in \uO$,
	\begin{align*}
	\lim_{n \to \infty}\limsup_{N \to \infty}\frac{1}{N+1}\sum_{k=0}^{N}
	{\underline \varrho}\big(\underline{Z}({\underline s}+ik\underline{h}, \uom, \uA;\ugb),\underline{Z}_n({\underline s}+ik\underline{h}, \uom, \uA;\ugb)\big)=0.
	\end{align*}
\end{lemma}

\begin{proof}
This is a generalization of Lemma 2 of \cite{RK-KM-Pal-2018}. 
In view of the definition of the metric ${\underline \varrho}$, it is sufficient to show
\begin{align}\label{limsupvarphi}
	\lim_{n \to \infty}\limsup_{N \to \infty}\frac{1}{N+1}\sum_{k=0}^{N}
	 \rho_{D_1}\big(\varphi(s_1+ikh_1),\varphi_n(s_1+ikh_1))\big)=0,
\end{align}
\begin{align}\label{limsupvarphiomega}
	\lim_{n \to \infty}\limsup_{N \to \infty}\frac{1}{N+1}\sum_{k=0}^{N}
	 \rho_{D_1}\big(\varphi(s_1+ikh_1,\omega_1),\varphi_n(s_1+ikh_1,\omega_1))\big)=0
\end{align}
for almost all $\omega_1$, and the corresponding results for periodic Hurwitz
zeta-functions.

The proof of \eqref{limsupvarphi} and \eqref{limsupvarphiomega} are included in the proof of
Lemma~3 of \cite{RK-KM-2017-Pal}, though the formulas themselves are not explicitly stated
there.
The corresponding results for periodic Hurwitz zeta-functions are Theorem 4.1 and
Theorem 4.4 of Laurin\v cikas and Macaitien\.e \cite{AL-RM-2009}.
\end{proof}

\medskip

In the third step, we introduce one more probability measure, for almost all $\uom \in \uO$, defined by
$$
P_{N,\uom}(A):=\frac{1}{N+1}\#\bigg\{0\leq k \leq N: \underline{Z}({\underline s}, \uom, \uA;\ugb)\in A\bigg\}, \quad A \in {\mathcal B}(\underline{H}).
$$


\begin{lemma}\label{RK-lem-3}
	Suppose that $L(\uA,\uhh)$ is as in Theorem~\ref{RK-KM-gen-lim}. Then the measures $P_N$ and $P_{N,\uom}$ both converge weakly to the same probability measure (denote it by $P$) on $(\underline{H},{\mathcal B}(\underline{H}))$ as $N \to \infty$.
\end{lemma}

\begin{proof}
This is a generalization of Lemma 3 of \cite{RK-KM-Pal-2018}.
Let $(\widetilde{\Omega},{\mathcal B}({\widetilde{\Omega}}),\widetilde{P})$ be a certain
probability space, $\theta_N:\widetilde{\Omega}\to\rr$ be a discrete random variable satisfying
$$\widetilde{P}(\theta_N=k)=(N+1)^{-1} \qquad k=0,\ldots, N,$$
and define the 
$\underline{H}$-valued random element $\underline{X}_{N,n}(\underline{s})$ 
on $\widetilde{\Omega}$ by
$$
\underline{X}_{N,n}(\underline{s})=\underline{Z}_n(\underline{s}+i\theta_N\underline{h},
\uA;\ugb).
$$
Then the distribution of $\underline{X}_{N,n}(\underline{s})$ is clearly $P_{N,n}$.

Let $\underline{X}_n(\underline{s})$ be an
$\underline{H}$-valued random element whose distribution is $P_n$.    Then
Lemma \ref{RK-lem-1} implies that $\underline{X}_{N,n}(\underline{s})$ converges to
$\underline{X}_n(\underline{s})$ in distribution, as $N\to\infty$.

We can show that the family $\{P_n\}$ is tight, in a standard way; see, for example,
pp.269-270 of \cite{RK-AL-2011}.    Therefore by Prokhorov's theorem (see \cite{PB-1968})
we can choose
a subsequence $\{P_{n_k}\}\subset\{P_n\}$, which converges weakly to a certain
probability measure $P$ on $(\underline{H},{\mathcal B}(\underline{H}))$ as $k\to\infty$.
That is, $\underline{X}_{n_k}$ converges to $P$ in distribution as $k\to\infty$.

Define another random element 
$\underline{X}_N(\underline{s})=\underline{Z}(\underline{s}+i\theta_N\underline{h},
\uA;\ugb)$.   Then for any $\varepsilon>0$,
\begin{align*}
&\lim_{n\to\infty}\limsup_{N\to\infty}\widetilde{P}(\underline{\varrho}(
\underline{X}_N(\underline{s}),\underline{X}_{N,n}(\underline{s}))\geq\varepsilon)\\
&=\lim_{n\to\infty}\limsup_{N\to\infty}\frac{1}{N+1}\#\{0\leq k\leq n: 
\underline{\varrho}(\underline{Z}(\underline{s}+i\theta_N\underline{h},
\uA;\ugb),\underline{Z}_n(\underline{s}+i\theta_N\underline{h},
\uA;\ugb))\geq\varepsilon\}\\
&\leq \lim_{n\to\infty}\limsup_{N\to\infty}\frac{1}{\varepsilon(N+1)}\sum_{ k=0}^n 
\underline{\varrho}(\underline{Z}(\underline{s}+ik\underline{h},
\uA;\ugb)-\underline{Z}_n(\underline{s}+ik\underline{h},
\uA;\ugb))=0
\end{align*}
by Lemma \ref{RK-lem-2}.
Therefore we can apply Theorem 4.2 of \cite{PB-1968} to find that
$\underline{X}_N(\underline{s})$ converges to $P$ in distribution, that is, $P_N$
converges weakly to $P$, as $N\to\infty$.
Moreover this fact implies that $P$ does not depend on the choice of the subsequence
$\{P_{n_k}\}$.    Therefore by Theorem 1.1.9 of \cite{AL-1996} we see that $P_n$
converges weakly to $P$ (as $n\to\infty$).

Finally put
$$
\widehat{\underline{X}}_{N,n}(\underline{s})=\underline{Z}_n(\underline{s}+i\theta_N\underline{h},\uom,
\uA;\ugb),\quad
\widehat{\underline{X}}_{N}(\underline{s})=\underline{Z}(\underline{s}+i\theta_N\underline{h},
\uom,\uA;\ugb)
$$
and argue as above.    It follows that 
$\widehat{\underline{X}}_{N,n}(\underline{s})\to \underline{X}_{n}(\underline{s})$
(as $N\to\infty$) in distribution.
We already mentioned above that $\underline{X}_n(\underline{s})\to P$ in distribution.
Also using Lemma \ref{RK-lem-2} we find
$$
\lim_{n\to\infty}\limsup_{N\to\infty}\widetilde{P}(\underline{\varrho}(
\widehat{\underline{X}}_N(\underline{s}),\widehat{\underline{X}}_{N,n}(\underline{s}))\geq\varepsilon) =0.
$$
Therefore again using Theorem 4.2 of \cite{PB-1968} we obtain 
that $P_{N,\uom}$ converges weakly to $P$ as $N\to\infty$.
\end{proof}

The final step of the proof is to identify the measure $P$ in  Lemma~\ref{RK-lem-3}. 

\begin{lemma}\label{RK-lem-4}
	The probability measure $P$ coincides with the probability measure $P_\uZ$.
\end{lemma}

\begin{proof}
This is a generalization of Lemma 4 of \cite{RK-KM-Pal-2018}.   Using Lemma \ref{ergodic}
and the classical Birkhoff-Khintchine theorem (see \cite{HCR-MRL-1967}),
we just mimic the standard argument (see, for example, the proof of Theorem 4 of
\cite{EB-AL-2015-LMJ}).
\end{proof}

The proof of Theorem~\ref{RK-KM-gen-lim} is completed.

From Theorem~\ref{RK-KM-gen-lim}, we can immediately deduce simpler discrete mixed joint limit theorem when the set $L(\uA,\uhh)$ is replaced by $L(\uA,h)$, 
that is the case $h_1=h_{21}=...=h_{2j}=h$.    Here we give the statement of such discrete mixed joint limit theorem.

\begin{theorem}\label{RK-KM-gen-lim-h}
	Suppose that $\varphi\in\mathcal{M}$, the elements of the set $L(\uA,h)$ are linearly independent over $\qq$ and $\text{rank}(B_j)=l(j)$, $j=1,...,r$. Then $P_N$ converges weakly to
	$P_{\uZ}$
	as $N \to \infty$.
\end{theorem}

\section{Proof of Theorems~\ref{RK-KM-general-h} and \ref{RK-KM-general}}

For the proof of universality property in the Voronin sense, we need to construct the support of the probability measure $P_\uZ$ in an explicit form.

Assume that $\varphi\in\widetilde{S}$, and let $K_1,K_{2jl},f_1(s), f_{2jl}$, $j=1,...,r$, $l=1,...,l(j)$,  be as in our most general universality result, i.e., Theorem \ref{RK-KM-general}.    Suppose that $M>0$ is a sufficiently large number such that
$K_1$ belongs to
$$
D_M=\big\{s \in \cc:\sigma_0<\sigma<1,|t|<M\big\}.
$$
Since $\varphi\in\widetilde{S}$, then
$D_{\varphi}=\{s \in \cc:\sigma>\sigma_0,\sigma\neq 1\}$. This gives us that $D_M\subset D_{\varphi}$. Moreover we can find $T>0$ such that each $K_{2jl}$ for $j=1,...,r$ and $l=1,...,l(j)$ is a part of
$$
D_T=\bigg\{s \in \cc: \ \frac{1}{2}<\sigma<1,\ |t|<T \bigg\}.
$$

Now we choose $D_1=D_M$ and $D_2=D_T$ and consider
an explicit form of the support $S_{\underline{Z}}$ of the probability measure $P_{\underline{Z}}$.
Let
$S_{\varphi}:=\{f \in H(D_M): f(s)\not =0 \ \text{for} \ D_M,\ \text{or} \ f(s) \equiv 0\}$.

\begin{lemma}\label{RK-support}
	Suppose that the elements of the set $L(\uA,\uhh)$ are linearly independent over $\qq$. Then the support of the measure $P_{\underline{Z}}$ is the set
	$S_{\underline{Z}}:=S_{\varphi}\times H^{\lambda}(D_T)$.
\end{lemma}

\begin{proof}
This is Lemma~5.2 of \cite{RK-KM-2017-BAMS}.   In \cite{RK-KM-2017-BAMS} no proof
is given, so here we present a proof (which is again standard; see, for example, Theorem 5
of \cite{JG-RM-SR-DS-2010}).

We want to find the minimal closed set $A$ which satisfies $P_{\underline{Z}}(A)=1$.
It is sufficient to consider the set of the form
$A=A_1\times A_{21}\times\cdots\times A_{2r}$, where $A_1\in{\mathcal B}(H(D_M))$ and
$A_{2j}\in{\mathcal B}(H(D_T)^{l(j)})$ ($1\leq j\leq r$).   Then
\begin{align*}
&P_{\underline{Z}}(A)=m_H(\uom\in\underline{\Omega}: \underline{Z}(\underline{s},\uom,
\uA;\ugb)\in A)\\
&\quad=m_{1H}(\omega_1\in\Omega_1:\varphi(s_1,\omega_1)\in A_1)\\
&\qquad\times \prod_{j=1}^r m_{2H}(\omega_{2j}:(\zeta(s_{2j1},\alpha_j,\omega_{2j};
\mathfrak{B}_{j1}),\ldots,\zeta(s_{2jl(j)},\alpha_j,\omega_{2j};\mathfrak{B}_{jl(j)}))
\in A_{2j}).
\end{align*}
The condition $P_{\underline{Z}}(A)=1$ implies that all factors on the right-hand side
are equal to 1.
The minimal closed set $A_1$ satisfying 
$m_{1H}(\omega_1\in\Omega_1:\varphi(s_1,\omega_1)\in A_1)=1$ is $S_{\varphi}$
(Lemma 5.12 of Steuding \cite{JST-2007}).    Under the assumption 
${\text{rank}}B_j=l(j)$, the minimal closed set $A_{2j}$ satisfying
$m_{2H}(\omega_{2j}:(\zeta(s_{2j1},\alpha_j,\omega_{2j};
\mathfrak{B}_{j1}),\ldots,\zeta(s_{2jl(j)},\alpha_j,\omega_{2j};\mathfrak{B}_{jl(j)}))
\in A_{2j})=1$ is $H(D_T)^{l(j)}$ for $j=1,\ldots,r$
(Laurin\v cikas \cite{AL-2007}; see also \cite{AL-SS-2009}).
Therefore the minimal $A$ is equal to $S_{\varphi}\times H(D_T)^{\lambda}$.
\end{proof}

\begin{proof}[Proof of Theorem~\ref{RK-KM-general} (and hence of Theorem~\ref{RK-KM-general-h}).]
	The proof of Theorem~\ref{RK-KM-general} we obtain by combining results of Theorem~\ref{RK-KM-gen-lim}, Lemma~\ref{RK-support} and  the Mergelyan theorem on the approximation of analytic functions by polynomials (see \cite{SNM-1952}). We omit the details since
the argument is standard and the same as in Section 4 of \cite{RK-KM-2017-Pal}.
\end{proof}


Finally, we would like to mention that probably it is possible to generalize our Theo\-rem~\ref{RK-KM-general} in one more direction. 	In this paper, we consider a wide collection of periodic Hurwitz zeta-functions $\zeta(s,\alpha_j;\gb_{jl})$, where we attach the collection of sequences $\{\gb_{jl}\}_{1\leq l\leq l(j)}$ to a parameter $\alpha_j$. 
However the common differences $h_{2j}$ is the same for all $\gb_{jl}$, 
$l=1,\ldots,l(j)$.
The new idea is to 
consider the situation when to each $\gb_{jl}$ the attached common difference $h_{2jl}$
may be different.


\end{document}